\documentclass[letter, 11 pt, reqno]{amsart}

\usepackage[margin=1in]{geometry}

\usepackage[hyphens]{url} 
\usepackage{amsthm}
\usepackage{graphicx}
\usepackage{tikz-cd}
\usepackage{bbm}
\usepackage{tasks}
\usepackage{fancyhdr}
\usepackage[T1]{fontenc}
\usepackage[utf8]{inputenc}
\usepackage{mathtools}
\usepackage{amsmath,amsfonts}
\usepackage[inline]{enumitem}
\usepackage{mathrsfs}
\usepackage{textcomp}
\usepackage{tikz-cd}
\usepackage{polski}
\usepackage{array}
\usepackage{amssymb}
\usepackage[new]{old-arrows}
\usepackage{kantlipsum}

\usetikzlibrary{positioning}
\usepackage{caption}
\usepackage{pgfplots}
\pgfplotsset{compat=1.18}
\usetikzlibrary{intersections}
\usetikzlibrary{patterns}
\usetikzlibrary{arrows}
\usetikzlibrary{decorations.markings}
\usepgfplotslibrary{fillbetween}
\usepackage{capt-of}

\tikzset{%
  dot/.style={circle, draw, fill=black, inner sep=0pt, minimum width=4pt},
  top/.style={anchor=south, inner sep=5pt},
}

\usepackage{hyperref}
\usepackage{cleveref}
\usepackage{moreenum}
\usepackage[english]{babel}
\usepackage[scr=boondox]{mathalfa}
\usepackage{aliascnt}
\usepackage{etoolbox}
\usepackage{empheq}
\usepackage{stackengine}
\usepackage{esvect}
\usepackage{array}
\usepackage{mleftright}
\usepackage{float}
\addto\extrasenglish{%
}

\hypersetup{
    colorlinks=true,
    linkcolor=blue,
    filecolor=black,      
    urlcolor=black,
    citecolor=red
}

\makeatletter
\patchcmd{\@maketitle}
  {\ifx\@empty\@dedicatory}
  {\ifx\@empty\@date \else {\vskip1ex \centering\footnotesize\@date\par\vskip1ex}\fi
   \ifx\@empty\@dedicatory}
  {}{}
\patchcmd{\@adminfootnotes}
  {\ifx\@empty\@date\else \@footnotetext{\@setdate}\fi}
  {}{}{}
\makeatother

\patchcmd{\section}{\scshape}{\bfseries}{}{}
\makeatletter
\renewcommand{\@secnumfont}
\makeatother

\makeatletter
\renewcommand\subsection{\@startsection{subsection}{2}%
  \z@{-.5\linespacing\@plus-.7\linespacing}{.5\linespacing}%
  {\normalfont\bfseries}}
\makeatother

\calclayout

\graphicspath{ {./images/} }

\theoremstyle{plain}
\newtheorem{thm}{Theorem}[section]

\newtheorem*{thm*}{Theorem}

\newaliascnt{prop}{thm}
    \newtheorem{prop}[prop]{Proposition}
    \aliascntresetthe{prop}

\newaliascnt{lem}{thm}
    \newtheorem{lem}[lem]{Lemma}
    \aliascntresetthe{lem}

\newaliascnt{cor}{thm}
    \newtheorem{cor}[cor]{Corollary}
    \aliascntresetthe{cor}

\newaliascnt{conj}{thm}
    \newtheorem{conj}[conj]{Conjecture}
    \aliascntresetthe{conj}

\theoremstyle{definition}

\newaliascnt{exm}{thm}
    \newtheorem{exm}[exm]{Example}
    \aliascntresetthe{exm}

\newaliascnt{dfn}{thm}
    \newtheorem{dfn}[dfn]{Definition}
    \aliascntresetthe{dfn}

\newaliascnt{rem}{thm}
    \newtheorem{rem}[rem]{Remark}
    \aliascntresetthe{rem}

\numberwithin{equation}{section}

\NewDocumentCommand{\evalat}{sO{\big}mm}{%
\IfBooleanTF{#1}
{\mleft. #3 \mright|_{#4}}
{#3#2|_{#4}}%
}

\DeclareMathOperator{\tr}{tr}

\usepackage{titletoc}
\title{An Invariant for Transverse Coassociative 4-Folds}
\author{Dylan Galt}

\begin{document}

\date{\small \today}

\address{Department of Mathematics, Harvard University, Cambridge, MA 02138}
\email{dgalt@math.harvard.edu}

\begin{abstract}
We define a $\mathbb{Z}_2$-valued invariant for transversely-intersecting coassociative $4$-folds equipped with spin structures. Our main result shows this invariant provides an obstruction to separating two such coassociatives through a family of transverse coassociative deformations. We further prove that there is a canonical generalized connected sum of two such transverse coassociatives whose diffeomorphism type is determined by our invariant. When one coassociative is a graph over the other, we relate our invariant to the parity function in near-symplectic geometry. Finally, we discuss conjectural consequences for non-compactness phenomena and compute our invariant for the $\textup{Sp}(1)$-invariant coassociatives discovered by Harvey and Lawson.
\end{abstract}

\maketitle

\dottedcontents{section}[0em]{\bfseries}{0em}{1pc}
\dottedcontents{subsection}[0em]{}{-2em}{1pc}
\setcounter{tocdepth}{2}
\tableofcontents

\section{Introduction}

The primary object of study in this paper is a special class of minimal submanifolds discovered by Harvey and Lawson called coassociative submanifolds \cite{HL82}. Coassociatives are $4$-dimensional calibrated submanifolds of $7$-manifolds possessing a $G_2$-structure and are of interest especially for the fundamental role they play in gauge theory on special holonomy spaces \cite{DS11}\cite{D17}. Many important questions about them can be formulated by analogy with special Lagrangian submanifolds of Calabi-Yau $3$-folds, and among these some of the most fundamental concern non-compactness phenomena. Singular degenerations of special Lagrangians have been studied extensively, including by Butscher \cite{B04}, Lee \cite{L03}, and Joyce \cite{J03}\cite{J03I}\cite{J04I}\cite{J04}\cite{J04II}, with one prototypical example of an isolated conical singularity given by the transverse intersection of two special Lagrangians. In contrast, the transverse intersection of two coassociatives $X_\pm$ in a $G_2$-manifold $Y$ is, if compact, a disjoint union of embedded circles in $Y$.

In a forthcoming paper \cite{G25}, we show under mild hypotheses that when $X_\pm$ are compact the singular coassociative $X_+\cup X_-\subset Y$ can be resolved and so in principle such singular degenerations can occur. A similar result for the non-compact $\textup{Sp}(1)$-invariant coassociatives in $\mathbb{R}^7$ discovered by Harvey and Lawson \cite{HL82} has been proven by Lotay and Kapouleas \cite{LK17}. In this paper, we study the question of when two transverse coassociatives can be separated and discuss the relation of our main results to such non-compactness phenomena. We explain how the separability of $X_\pm$ is related to the moduli space of deformations of the smooth coassociative $X$ resolving $X_+\cup X_-$ and conjecture the existence of a ``coassociative blow-down'' for $X$ (see \autoref{c4.3}). We also develop connections with near-symplectic geometry, emphasizing the perspective that two transverse coassociatives, when one is a graph over the other, behaves like a ``nonlinear'' analogue of a near-synplectic form. We hope this will stimulate research deepening the analogy.

\subsection{Main results} Let $Y$ be a smooth $7$-manifold equipped with a $G_2$-structure $\varphi$ and let $X_\pm\subset Y$ be coassociative submanifolds intersecting transversely with compact intersection $Z$. Our main theorem can be stated as follows.

\begin{thm}\label{t1.1}
If $X_\pm$ are equipped with spin structures, there is a map
\begin{equation*}
    \mu: \pi_0(Z)\rightarrow \{\pm 1\}
\end{equation*}
with the following properties:
\begin{enumerate}
    \item $\mu$ is invariant under coassociative deformations $(X_\pm(t))_{t\in [0,1]}$ that stay transverse for all $t\in [0,1]$ and oriented diffeomorphisms of $Y$ preserving $\varphi$.
    \item If a component $\Gamma\in \pi_0(Z)$ is locally removable through a family of coassociative deformations $(X_\pm(t))_{t\in [0,1]}$, then $\mu(\Gamma)=-1$.
\end{enumerate}
\end{thm}

We will make precise what is meant by locally removable, but suffice it to say for the moment that this means roughly \textit{removable along discs}. The main point of \autoref{t1.1} is that the map $\mu$ provides an obstruction to coassociative separability. After proving \autoref{t1.1}, we compute the invariant $\mu$ for the family of examples studied by Lotay and Kapouleas \cite{LK17} involving the $\textup{Sp}(1)$-invariant coassociative submanifolds of $\mathbb{R}^7$ constructed by Harvey and Lawson \cite{HL82}. The gluing constructions \cite{LK17}\cite{G25} produce smooth coassociatives diffeomorphic to the (generalized) connected sum $X_+\#_Z X_-\subset Y$ of $X_\pm$ along their intersection $Z$ and our second main result relates its topology to our invariant $\mu$.

\begin{thm}\label{t1.2}
There is a canonical generalized connected sum $X_+\#_Z X_-$ of $X_\pm$ and if $X_\pm$ are simply-connected and spin it is diffeomorphic to
\begin{equation*}
    X_+\# X_-\# W_\mu\#(k-1)(S^1\times S^3)
\end{equation*}
where $k=|\pi_0(Z)|$ and
\begin{equation*}
    W_\mu = \mu_{+}(S^2\times S^2)\#\mu_{-}(\mathbb{CP}^2\#\overline{\mathbb{CP}}^2), \;\;\; \mu_{\pm}=|\{\Gamma\in \pi_0(Z)\; : \; \mu(\Gamma)=\pm 1\}|.
\end{equation*}    
\end{thm}

We also discover that when $X_\pm$ are graphical near their intersection, the invariant $\mu$ admits a different description, inspired by the parity invariant in near-symplectic geometry \cite{P07}. In light of this, \autoref{t1.1} can be viewed as a sort of local analogue of Gompf's theorem (c.f. Theorem 1.8 in \cite{P07}) on the number (modulo two) of ``even'' zero circles of a near-symplectic form on a closed, oriented $4$-manifold. Finally, we explain why it is likely not possible in general to attach signs to the intersection circles of $X_\pm$ in a meaningful way without extra assumptions on their geometry (see \autoref{r2.18} and \autoref{t3.11}).
\vspace{2mm}

\noindent
\textbf{Acknowledgements.} I am grateful to my Ph.D. supervisor, Simon Donaldson, for suggesting this problem and sharing generously his ideas, especially with regard to \autoref{c4.3}. Thank you to Olga Plamenevskaya, Arunima Ray, Dennis Sullivan, and Yunpeng Niu for useful conversations. I am especially grateful to Jason Lotay for his 2017 lecture ``Invariant Coassociative 4-folds via Gluing'' about joint work with Nikolaos Kapouleas, from which I learned many aspects of coassociative geometry relevant to this paper. 
\vspace{2mm}

\noindent
\textit{This work was partially supported by the Simons Collaboration on Special Holonomy in Geometry, Analysis, and Physics and by National Science Foundation DMS grant 2503365. This article comprised part of the author's Ph.D. thesis.}

\section{The Invariant}\label{s02}
In \autoref{s2.1} we provide background material in $G_2$ geometry. In \autoref{s2.2} we define our invariant and prove \autoref{t1.2}. We describe our main motivating class of examples in \autoref{s2.4} and then prove our main result \autoref{t1.1} in \autoref{s2.3}. In \autoref{s02.5} we study our invariant for graphical coassociatives and relate it to near-symplectic geometry.

\subsection{$G_2$-manifolds and coassociative submanifolds}\label{s2.1}
Let $\mathbb{O}$ denote the octonions, which form a normed division algebra over $\mathbb{R}$. The imaginary octonions $\textup{Im}(\mathbb{O})$ form a seven-dimensional real vector space and octonionic multiplication induces a cross product \cite{HL82} on $\textup{Im}(\mathbb{O})$ given by $x\times y := \textup{Im}(\bar{y}x)$. Using this, one can define an alternating, trilinear form
\begin{equation}\label{e2.1}
    \varphi_0(x,y,z) :=\left< x,y\times z\right>.
\end{equation}
The subgroup of $\textup{GL}_7(\mathbb{R})$ stabilizing $\varphi_0$ is isomorphic to $G_2$, the compact real form of the Lie group corresponding to the exceptional complex Lie algebra $\mathfrak{g}_2$ \cite{B05}. The $\textup{GL}_7(\mathbb{R})$ orbit of $\varphi_0$ in $\Lambda^3(\mathbb{R}^7)^\ast$ is an open set \cite{B05}, denoted $\Lambda^3_+(\mathbb{R}^7)^\ast$, and on a $7$-dimensional smooth manifold $Y$ we have a corresponding subbundle $\Lambda^3_+T^\ast Y$ of $\Lambda^3T^\ast Y$.

\begin{dfn}
We say that $Y$ is a $G_2$-manifold if there is a section $\varphi\in \Omega^3_+(Y)$. In this case, we call $\varphi$ a $G_2$-structure on $Y$.
\end{dfn}

The term $G_2$-manifold is sometimes reserved for a $7$-manifold equipped with a torsion-free $G_2$-structure. We caution the reader that this is not our usage of the term. The existence of such a 3-form $\varphi$ is equivalent \cite{B05} to a reduction of the structure group of the frame bundle of $Y$ to $G_2$. Since $G_2\subset \textup{SO}(7)$, a $G_2$-structure $\varphi$ on $Y$ determines a Riemannian metric $g_\varphi$ and thus a dual $4$-form $\ast\varphi\in \Omega^4(Y)$. If $\varphi$ is co-closed then $\ast \varphi$ is a calibration in the sense of Harvey and Lawson \cite{HL82}. An orientable 4-dimensional submanifold $X\subset Y$ is calibrated by $\ast\varphi$ (up to a choice of orientation of $X$) if and only if $\varphi|_X\equiv 0$ \cite{J00}. Thus, we adopt a slightly non-standard weakening of the definition of coassociative in \cite{HL82}.

\begin{dfn}\label{d2.7}
We call an oriented submanifold $X^4\subset Y$ coassociative if $\varphi|_X\equiv 0$.
\end{dfn}

This condition is overdetermined and one must in general require that $\varphi$ be closed in order for coassociatives to exist even locally \cite{J00}. However, we will use in this paper only the consequences of the isotropic condition in \autoref{d2.7} and thus will neither require that $\varphi$ be closed nor co-closed.

\subsection{Spin structures and generalized connected sum}\label{s2.2}
Let $X_\pm$ be two transversely-intersecting coassociative submanifolds of a $G_2$-manifold $(Y,\varphi)$ with compact intersection $Z$. Let $N$ denote the normal bundle to $Z$ in $Y$ and let $N_\pm$ denote the normal bundle to $Z$ in $X_\pm$. We can form the (generalized) connected sum $X_+\#_Z X_-$ of $X_+$ and $X_-$ along $Z$ as in \cite{K07}. First, remove tubular neighborhoods $T_\pm$ of $Z$ in $X_+$ and $X_-$. Then, fix an isomorphism $h$ between $N_+$ and $N_-$ that reverses orientation on the fibres. Finally, form the space $X_+\#_Z X_-$ by identifying elements of $T_+\setminus Z$ with elements of $T_-\setminus Z$ via $h$. Note that this is naturally a submanifold of $Y$. The diffeomorphism type of $X_+\#_Z X_-$ depends on the choice of $h$, and in general there are two such choices because $\pi_1(\textup{SO}(3))=\mathbb{Z}_2$, but in our case we have the following.

\begin{lem}\label{l03.1}
The $G_2$-structure $\varphi$ determines an isomorphism between $N_+$ and $N_-$ and thus a canonical generalized connected sum $X_+\#_Z X_-$.
\end{lem}
\begin{proof}
Since $X_+$ and $X_-$ intersect transversely we have a splitting
\begin{equation*}
    \Lambda^2 N^\ast=\Lambda^2 N^\ast_+\oplus (N^\ast_+\otimes N^\ast_-)\oplus\Lambda^2 N^\ast_-.
\end{equation*}
Since $X_+$ and $X_-$ are coassociative, the components of the section $\varphi||_Z\;\lrcorner\; u\in \Lambda^3 T^\ast Y|_Z$ in $\Lambda^2 N^\ast_\pm$ vanish, so that
\begin{equation*}
    \varphi||_Z\;\lrcorner\; u\in \Gamma(Z,N^\ast_+\otimes N^\ast_-)
\end{equation*}
where the symbol $||$ is used to denote the restriction of $\varphi$ as a bundle section. Because $\varphi$ is strongly nondegenerate \cite{MS10}, this section induces a bundle isomorphism between $N_+$ and $N_-$. The isomorphism determined by $-(\varphi||_Z\;\lrcorner\; u)$ is fibrewise orientation-reversing and so determines a canonical generalized connected sum $X_+\#_Z X_-$.
\end{proof}

We now explain how to define our invariant, using the isomorphism between $N_\pm$ given by \autoref{l03.1}. Recall that by a framing of a vector bundle one means a choice of trivialization up to homotopy. On an oriented vector bundle over the circle of rank at least $3$, framings are in bijection with spin structures \cite{KT91}. If $X_+$ is spin, the spin structure on $TX_+$ induces a spin structure on $TX_+|_Z$, which in turn induces a preferred framing. We have the orthogonal splitting $TX_+|_Z = TZ\oplus N_+$ using the metric $g_\varphi$ to regard the normal bundle $N_+$ as a subbundle of $TX_+|_Z$. Since $Z$ is oriented, the framing for $TX_+|_Z$ induces a framing for $N_+$. Likewise, a spin structure on $X_-$ induces a framing for $N_-$. 



\begin{dfn}\label{d4.7}
Given spin structures on $X_\pm$, define
\begin{equation*}
    \mu_{X_+,X_-}:\pi_0(Z)\rightarrow\mathbb{Z}_2
\end{equation*}
by setting $\mu_{X_+,X_-}(\Gamma)=+1$ if the framings for $N_\pm|_\Gamma$ agree when compared by the isomorphism between $N_\pm$ determined by the $G_2$-structure. Otherwise, set $\mu_{X_+,X_-}(\Gamma)=-1$.
\end{dfn}

The map $\mu_{X_+,X_-}$ in \autoref{d4.7} is an invariant of the pair $X_\pm$ in the following sense.

\begin{prop}\label{p02.6}
The map $\mu_{X_+,X_-}$ satisfies the following properties:
\begin{enumerate}
    \item If $f:Y\rightarrow Y$ is an oriented diffeomorphism with $f^\ast\varphi=\varphi$, then we have for all $\Gamma\in \pi_0(Z)$ that
    \begin{equation*}
    \mu_{f(X_+),f(X_-)}(f(\Gamma))=\mu_{X_+,X_-}(\Gamma)
    \end{equation*}
    \item If $\varphi_t$ is a smooth family of $G_2$-structures on $Y$ and $F_\pm : X_\pm \times [0,1]\rightarrow Y$ is a smooth family of deformations such that $X_\pm(t)=F_\pm(X_\pm\times\{t\})$ are coassociative with respect to $\varphi_t$ and intersect transversely for all $t\in [0,1]$ then
    \begin{equation*}
        \mu_{X_+(0),X_-(0)}(\Gamma)=\mu_{X_+(t),X_-(t)}(F(\Gamma\times\{t\}))
    \end{equation*}
    for all $\Gamma\in \pi_0(X_+(0)\cap X_-(0))$ and $t\in [0,1]$.
\end{enumerate}
\end{prop}
\begin{proof}
A diffeomorphism $f:Y\rightarrow Y$ will induce an isomorphism between the normal bundles $N_{\Gamma/X_\pm}$ and $N_{f(\Gamma)/f(X_\pm)}$, and if $f^\ast\varphi=\varphi$ then this will commute with the isomorphism determined by the $G_2$-structure, which is used to define the map $\mu$. This justifies $(1)$. The map $F_\pm$ induces a homotopy between the spin framings for $N_\pm(0)$ and $N_\pm(t)$, which gives $(2)$.
\end{proof}

The following result characterizes the diffeomorphism type of the canonical generalized connected sum $X_+\#_Z X_-\subset Y$ from \autoref{l03.1} in terms of the invariant $\mu_{X_+,X_-}$.

\begin{thm}\label{t4.11}
Suppose that $X_+$ and $X_-$ are simply-connected and spin. The generalized connected sum $X_+\#_Z X_-$ is diffeomorphic to 
\begin{equation*}
    X_+\# X_-\# W_\mu\#(k-1)(S^1\times S^3)
\end{equation*}
where $k=|\pi_0(Z)|$ and
\begin{equation*}
    W_\mu = \mu_{+}(S^2\times S^2)\#\mu_{-}(\mathbb{CP}^2\#\overline{\mathbb{CP}}^2), \;\;\; \mu_{\pm}=|\{\Gamma\in \pi_0(Z)\; : \; \mu_{X_+,X_-}(\Gamma)=\pm 1\}|.
\end{equation*}
\end{thm}
\begin{proof}
Suppose $Z$ is connected. Since $X_+$ is simply-connected, $Z$ bounds an immersed disc with only isolated ordinary double point singularities in its interior. Since $X_+$ has dimension 4 and no boundary, we can use a sequence of finger moves to push the self-intersection points off the disc. The double points can be connected to the boundary of the disc by pairwise disjoint arcs. At each double point, we have two branches of the disc, and we can push one branch along the corresponding arc until the self-intersection point is removed. Likewise for $X_-$. Shrinking these embedded discs radially, we can assume $Z$ is contained in standard 4-balls in $X_\pm$. Thus, $X_\pm\cong X_\pm\# S^4$ with $Z=S^1$ contained in these copies of $S^4$, so that $X_+\#_Z X_-$ is diffeomorphic to $X_+\# X_-\# (S^4\#_{S^1}S^4)$.

We now show that $S^4\#_{S^1}S^4$ is diffeomorphic either to $S^2\times S^2$ or the nontrivial $S^2$ bundle over $S^2$ (which is diffeomorphic to $\mathbb{CP}^2\#\overline{\mathbb{CP}}^2$), depending on the parity of $Z$. View each copy of $S^4$ as the boundary of a copy of $D^5$. We can view the second 5-ball $D^5=D^2\times D^3$ as a 2-handle that we attach to the boundary of the first $D^5$ along $S^1$, giving a manifold $W$ with $\partial W=S^4\#_{S^1}S^4$. This handle attachment requires an identification of $(\partial D^2)\times D^3$, the framing coming from the framing of $S^1$ in the second copy of $S^4=\partial (D^2\times D^3)$, with the normal bundle of $S^1$ in the first copy of $S^4$. This identification is the relative framing determined by the $G_2$-structure, as in \autoref{l03.1}. Our embedded discs glue together to give an embedded $S^2$ with equator the original $S^1$ and $W$ can be viewed as a $D^3$ bundle over this $S^2$ (c.f. Corollary 6.5 in \cite{K07} and \autoref{f7} below). Thus, $\partial W=S^4\#_{S^1}S^4$ is an $S^2$ bundle over $S^2$, of which there are two diffeomorphism types, distinguished by the parity of the relative framing, which is by definition the parity of the invariant $\mu$.

If $Z$ is disconnected, we continue inductively for the rest of its components, observing that the generalized connected sum $X_+\#_{\Gamma_1\sqcup\Gamma_2}X_-$ along two components $\Gamma_1$ and $\Gamma_2$ of $Z$ can be obtained from $(X_+\#_{\Gamma_1} X_-)$, the key point being that removing a neighborhood of $\Gamma_2$ from $(X_+\#_{\Gamma_1} X_-)$ leaves a connected manifold. Indeed, we can obtain $X_+\#_{\Gamma_1\sqcup\Gamma_2}X_-$ by taking the connected sum of $(X_+\#_{\Gamma_1} X_-)=(X_+\#_{\Gamma_1} X_-)\# S^4$ along $\Gamma_2\subset S^4$ with another copy of $S^4$ to obtain $(X_+\#_{\Gamma_1} X_-)\# (S^4\#_{\Gamma_2} S^4)$, then removing two points from this manifold and attaching a 1-handle to the corresponding manifold with boundary. Attaching this 1-handle is the same as taking the connected sum of $(X_+\#_{\Gamma_1} X_-)\# (S^4\#_{\Gamma_2} S^4)$ with $S^1\times S^3$. Proceeding in this way gives the result.
\end{proof}

The construction described in the proof of \autoref{t4.11} is depicted below in \autoref{f7} in the case when $X_\pm = S^4$ and $Z=S^1$.

\begin{figure}[h]
\centering
\begin{tikzpicture}[scale=0.9, decoration = {markings}]
  \draw (0,0) circle [radius=2];
  \draw[dashed] plot [smooth, tension=2] coordinates { (0,2) (3,3) (2,0)};
  \draw plot [smooth, tension=2] coordinates { (-1,1.75) (3.5,3.5) (1.75,-1)};
  \draw plot [smooth, tension=2] coordinates { (0.55,1.92) (2.4,2.4) (1.93,0.5)};
  \draw[dashed] plot [smooth, tension=1] coordinates { (0,2) (0.25,0.25) (2,0)};
  \node at (2,0) {\(\bullet\)};
  \node at (0,2) {\(\bullet\)};
  \node[label=left: ${\partial D^5=S^4_-}$] at (-1,-2) {};
  \node[label=left: ${\partial (D^2\times D^3)=S^4_+}$] at (7.5,1) {};
\end{tikzpicture}   
\caption[Constructing $S^4\#_{S^1}S^4$ via handle attachment]{Constructing $S^4\#_{S^1}S^4$ via a handle attachment. The dashed lines represent discs $D^2_\pm$ bounding the $S^1\subset S^4_\pm$, which have been pushed into the respective $5$-balls. The generalized connected sum $S^4\#_{S^1}S^4$ is the boundary of this picture, hence the boundary of a $D^3$ bundle over the $S^2$ formed by the two dashed lines.}
\label{f7}
\end{figure}

By a well-known result of Wall, $(S^2\times S^2) \# \mathbb{CP}^2$ and $(\mathbb{CP}^2\#\overline{\mathbb{CP}}^2)\#\mathbb{CP}^2$ are diffeomorphic. Using this, we obtain the following corollary of \autoref{t4.11}.

\begin{cor}
If $\mu_-\geq 1$, then $X_+\#_Z X_-$ is diffeomorphic to
\begin{equation*}
    X_+\# X_- \# k(\mathbb{CP}^2\#\overline{\mathbb{CP}}^2)\#(k-1)(S^1\times S^3)
\end{equation*}
\end{cor}

In the next section, we describe in detail our main motivating examples of transversely-intersecting coassociatives.

\subsection{The Harvey-Lawson coassociatives}\label{s2.4}
In their original paper \cite{HL82} on calibrated geometry, Harvey and Lawson constructed two different topological types of $\textup{Sp}(1)$-invariant coassociative submanifolds of $\textup{Im}(\mathbb{O})$. The author learned from a lecture \cite{LK17} of Lotay, concerning joint work of Lotay and Kapouleas, that these can be used to construct examples of transversely-intersecting coassociatives. We now describe these examples in detail, following \cite{LK17}.

Recall the standard $G_2$-structure $\varphi_0$ on $\textup{Im}(\mathbb{O})$ given in \autoref{e2.1}. Consider the splitting $\mathbb{O}=\mathbb{H}\oplus\mathbb{H}e$ so that $\textup{Im}(\mathbb{O})=\textup{Im}(\mathbb{H})\oplus\mathbb{H}e$. Take the copy of $\textup{Sp}(1)$ given by the unit quaternions $\textup{Sp}(1)=S^3\subset \mathbb{H}$ in the first copy of $\mathbb{H}$ with respect to our splitting. This acts on $\textup{Im}(\mathbb{O})=\textup{Im}(\mathbb{H})\oplus\mathbb{H}e$ by
\begin{equation}\label{e5.1}
    Q: \textup{Sp}(1)\times \textup{Im}(\mathbb{O})\rightarrow \textup{Im}(\mathbb{O}), \;\;\; q\cdot (x,ye)=(qx\bar{q},y\bar{q}e),
\end{equation}
preserving $\varphi_0$ and acting on $\textup{Im}(\mathbb{H})$ via the double cover $\textup{SU}(2)\rightarrow\textup{SO}(3)$. Define the following parameter space
\begin{equation}
    \mathcal{C}:=\{(\epsilon,c)\in \textup{Im}(\mathbb{H})\times\mathbb{R}_{\geq 0}\;\; :\;\; |\epsilon|=1\}
\end{equation}
and denote by $\mathcal{C}^+\subset\mathcal{C}$ the subspace where $c>0$. Harvey and Lawson use $\textup{Sp}(1)$ symmetry to reduce the coassociative condition to an ODE, which they then solve to obtain the following result.

\begin{thm}[\cite{HL82}, Thm. 3.2]
For each $(\epsilon,c)\in\mathcal{C}$ there is a corresponding coassociative submanifold $ \mathbf{M}_{\epsilon,c}\subset\textup{Im}(\mathbb{O})$ given by
\begin{equation*}
     \mathbf{M}_{\epsilon,c}:=\{(sq\epsilon\bar{q},r\bar{q}e)\;\;|\;\; q\in\textup{Sp}(1),\; s(4s^2-5r^2)^2=c, \;(s,r)\in\mathbb{R}_{>0}\times\mathbb{R}_{>0}\}.
\end{equation*}
\end{thm}

For fixed $(\epsilon,c)\in\mathcal{C}$ with $c>0$, there are two branches of solutions to the equation $s(4s^2-5r^2)^2=c$ depending on the sign of $4s^2-5r^2$. Thus, $ \mathbf{M}_{\epsilon,c}$ is the disjoint union of two connected coassociative submanifolds of $\textup{Im}(\mathbb{O})$, which we denote by
\begin{equation*}
    \mathbf{M}_{\epsilon,c}=M^+_{\epsilon,c}\cup M^-_{\epsilon,c}.
\end{equation*}
For $c>0$, the coassociative $M^+_{\epsilon,c}$ is diffeomorphic to the total space of the $\mathcal{O}(-1)$ line bundle over $\mathbb{P}^1$, while $M^-_{\epsilon,c}$ is diffeomorphic to $\mathbb{R}_{>0}\times S^3$ \cite{LK17}. In particular, for any $(\epsilon,c)\in\mathcal{C}^+$ the coassociative $M^-_{\epsilon,c}$ is spin while $M^+_{\epsilon,c}$ is not. When $c=0$, we get three branches of solutions: the coassociative 4-plane corresponding to $s=0$, which we denote by $H_0=\{0\}\oplus\mathbb{H}e$, and the so-called Lawson-Osserman cone \cite{LO77} given by
\begin{equation*}
    M^\pm_{\epsilon,0}:=\{(\pm\frac{\sqrt{5}}{2}rq\epsilon\bar{q},r\bar{q}e)\; | \; q\in \textup{Sp}(1),\; r\in \mathbb{R}_{>0}\}.
\end{equation*}

\begin{rem}\label{l02.10}
For $(\epsilon,c)\in \mathcal{C}^+$, the coassociative $M^+_{\epsilon,c}$ is asymptotically conical with one end while $M^-_{\epsilon,c}$ has one conical end and one flat end (c.f. \autoref{f4.3} below) \cite{LK17}. The asymptotic cone of $M^+_{\epsilon,c}$ is $M^+_{\epsilon,0}$ while the asymptotic cone of $M^-_{\epsilon,0}$ is $H_0\cup M^-_{\epsilon,0}$, which follows from the fact that rescaling $M^\pm_{\epsilon,c}$ by a factor of $\lambda>0$ gives $\lambda\cdot M^\pm_{\epsilon,c}=M^\pm_{\epsilon,\lambda^5c}$ \cite{L090}.
\end{rem}


Following \cite{LK17}, our examples of transversely-intersecting coassociatives will come by intersecting $\mathbf{M}_{\epsilon,c}$ with coassociative $4$-planes given by the translates $H_{s\epsilon}:=s\epsilon\oplus\mathbb{H}e$. We now give a proof of the following observation, which the author learned from \cite{LK17}.

\begin{lem}\label{l3.8}
Take any $(\epsilon,c)\in\mathcal{C}^+$. For $s>s_0(c):==(c/16)^{1/5}$, the coassociatives $M^+_{\epsilon,c}$ and $H_{s\epsilon}$ intersect transversely in a circle; for $s=s_0(c)$, they intersect non-transversely in a single point; and for $s<s_0(c)$ they are disjoint. The coassociatives $M^-_{\epsilon,c}$ and $H_{s\epsilon}$ intersect transversely in a circle for $s>0$ and are disjoint for $s\leq 0$.
\end{lem}
\begin{proof}
The calculation is the same for any $\epsilon$ and so we fix $\epsilon=e_2$. Given $q\in\textup{Sp}(1)$, computation shows that $qe_2\bar{q}=e_2$ if and only if $q=x_1e_1+x_2e_2$ with $x^2_1+x^2_2=1$. In particular, we see that
\begin{equation*}
    M^+_{e_2,c}\cap H_{se_2}=\{(se_2,r\bar{q}e)\;\;|\;\;q=x_1e_1+x_2e_2,\; x^2_1+x^2_2=1,s(4s^2-5r^2)^2=c\;\}.
\end{equation*}
For $M^+_{e_2,c}$ and fixed $s>0$ we are considering solutions $r$ to $5r^2=4s^2-\sqrt{c/s}$. There is a unique solution $r>0$ when $c<16s^5$. When $c=16s^5$ we must have $r=0$ to solve this equation, in which case we can see that $M^+_{e_2,c}$ and $H_{se_2}$ intersect only at $(se_2,0\cdot e)$ for $s$ solving $c=16s^5$. For $c>16s^5$ there are no solutions and so the coassociatives do not intersect. For $M^-_{e_2,c}$, we are considering solutions $r$ to $5r^2=4s^2+\sqrt{c/s}$, which exist if and only if $s>0$.
\end{proof}

The coassociatives $M^\pm_{\epsilon,c}$ can also be viewed as the total spaces of fibre bundles over the unit sphere $S^2\subset \textup{Im}(\mathbb{H})$, where the surface fibres are diffeomorphic to
\begin{equation*}
    \Sigma^+=\mathbb{C}\;\;\;\textup{and}\;\;\;\Sigma^-=\mathbb{R}_+\times S^1,
\end{equation*}
respectively, and can be viewed as surfaces of revolution over the corresponding profile curves in the $(s,r)$ parameter space \cite{F07}. Indeed, consider the projection
\begin{equation*}
    \pi: M_{\epsilon,c}\rightarrow\textup{Im}(\mathbb{H}), \;\;\; \pi(x,y)=x/|x|.
\end{equation*}

Since $\textup{SO}(3)$ acts transitively on the unit sphere in $\textup{Im}(\mathbb{H})$, for any unit vector $\epsilon'\in\textup{Im}(\mathbb{H})$ we can find some $q_0\in\textup{Sp}(1)$ such that $q_0\epsilon\bar{q}_0=\epsilon'$, so $\pi^{-1}(\epsilon')$ is given by
\begin{equation*}
    \{(sq\epsilon\bar{q},r\bar{q}e)\;\;|\;\; q\in\textup{Stab}_{\textup{Sp}(1)}(\epsilon')\cdot q_0, \;\; s(4s^2-5r^2)^2=c\}.
\end{equation*}
The stabilizer of $\epsilon'$ is isomorphic to $U(1)$. In the first case we get a family of circles for all $s>s_0(c)$ foliating $\Sigma^+$ except for the extra point, corresponding to $s=s_0(c)$, which is $(s_0(c)\epsilon',0\cdot e)$. In the second case, we get a 1-parameter family of circles foliating $\Sigma^-$ whose diameters go to infinity as $s\rightarrow +\infty$ and likewise as $s\rightarrow 0$. In both cases, these are precisely the intersection circles $M^\pm_{\epsilon,c}\cap H_{s\epsilon}$, as depicted in \autoref{f4.3} below.
\vspace{2mm}

\begin{figure}[h]
\begin{center}
\begin{tikzpicture}[decoration = {markings}]
\draw plot [smooth, tension=1.5] coordinates { (-1.5,2.5) (0,0) (1.5,2.5)};
\draw plot [smooth, tension =1.5] coordinates { (4,2.5) (5,1) (2.5,0) };
\draw plot [smooth, tension =1.5] coordinates { (8,2.5) (7,1) (9.5,0) };
    \node[label=left: $\Sigma^+$] at (-1.5, 2) {};
    \node[label=left: $\Sigma^-$] at (9,0.5) {};
\draw plot [smooth, tension =1] coordinates { (-1.25,1.39) (0,1) (1.25,1.39)};
\draw[dashed] plot [smooth, tension=1] coordinates { (-1.25,1.39) (0,1.5) (1.25,1.39) };
\draw plot [smooth, tension =1] coordinates { (5.13,1.5) (6,1.25) (6.85,1.5)};
\draw[dashed] plot [smooth, tension =1] coordinates { (5.13,1.5) (6,1.7) (6.85,1.5)};
\end{tikzpicture}   
\end{center}
\caption[Harvey-Lawson $\textup{Sp}(1)$-invariant coassociatives in $\mathbb{R}^7$]{The fibres $\Sigma^\pm$ with the intersection circles of $M^\pm_{\epsilon,c}$ and $H_{s\epsilon}$ depicted. The intersection circle bounds a disc in the fibre $\Sigma^+$.}
\label{f4.3}
\end{figure}

We now compute our invariant $\mu$ for the spin coassociatives $M^-_{\epsilon,c}$ and $H_{s\epsilon}$.

\begin{prop}\label{p02.11}
For any $(\epsilon,c)\in \mathcal{C}^+$ and $s>0$, we have
\begin{equation*}
    \mu(M^-_{\epsilon,c}\cap H_{s,\epsilon})=-1.
\end{equation*}
\end{prop}
\begin{proof}
Let $\Gamma$ denote the curve in the $(s,r)$-parameter space corresponding to $M^-_{\epsilon,c}$. Consider the smooth map $\pi: M^-_{\epsilon,c}\rightarrow \Gamma$ given by $\pi(x,y)=(|x|,|y|)$. The fibre $\pi^{-1}(s,r)$ can be identified with the $\textup{Sp}(1)$ orbit of the point $(s\epsilon, re)$. At $p\in M^-_{\epsilon,c}$ we have an identification
\begin{equation}\label{e02.5}
    T_{p}M^-_{\epsilon,c}=T(\textup{Sp}(1)\cdot p)\oplus T_{\pi(p)}\Gamma
\end{equation}
where $\textup{Sp}(1)\cdot p$ denotes the $\textup{Sp}(1)$ orbit of $p$. To compute the tangent space, we linearize the action $Q$ from \autoref{e5.1}. Let $q=e^{tv}$ for $v\in\mathfrak{sp}(1)$, so that $\bar{q}=e^{-tv}$, and
\begin{equation*}
    \frac{d}{dt}\bigg\rvert_{t=0}e^{tv}xe^{-tv}=[v,x], \;\;\; \frac{d}{dt}\bigg\rvert_{t=0}ye^{-tv}=-yv,
\end{equation*}
from which it follows that $\partial Q_{\textup{Id}}(v)(x,ye)=([v,x],-yve)$. Since $\textup{Sp}(1)=S^3\subset\mathbb{H}$, we have $\mathfrak{sp}(1)=T_{\textup{Id}}\textup{Sp}(1)=T_{e_1}S^3$, which is the span of $e_2,e_3,e_4$. Thus, we can write
\begin{equation}\label{e5.4}
    T(\textup{Sp}(1)\cdot p)=\{\partial Q_{\textup{Id}}(v)\cdot p\;\;|\;\; v\in\textup{span}_\mathbb{R}(e_2,e_3,e_4) \}.
\end{equation}
By \autoref{l3.8}, $M^-_{\epsilon,c}$ and $H_{s\epsilon}$ intersect transversely in a circle for any $s>0$ and so $\mu(M^-_{\epsilon,c}\cap H_{s,\epsilon})$ is well-defined. By \autoref{p02.6}, its value does not depend on $(\epsilon,c)\in \mathcal{C}^+$ or $s>0$, so we assume $\epsilon=e_2$ and $s=r=c=1$. Let $M=M^-_{e_2,1}$, $H=H_{e_2}$, and $Z=M^-_{e_2,1}\cap H_{e_2}$. Calculation shows $Z=\{(e_2,\cos(\theta)e_5-\sin(\theta)e_6)\; | \; \theta\in [0,2\pi)\}$. Let $p_\theta=(e_2,\cos(\theta)e_5-\sin(\theta)e_6)$. From \autoref{e5.4} we obtain
\begin{equation*}
   T(\textup{Sp}(1)\cdot p_\theta)=\{([v,e_2],(-\cos(\theta)e_5+\sin(\theta)e_6)v)\;\;|\;\; v\in\textup{span}_\mathbb{R}(e_2,e_3,e_4)\},
\end{equation*}
which is the span of $v_0,v_1,v_2$ with $v_0 = \sin\theta e_5+\cos\theta e_6$, $v_1 = -e_4+\cos\theta e_7-\sin\theta e_8$, and $v_2 = e_3 + \sin\theta e_7+\cos\theta e_8$. Here, $-v_0$ is an oriented basis for $T_{p_\theta}Z$. On the other hand, $T_{(1,1)}\Gamma$ is the span of $v_3=(4/3)e_2+(\cos(\theta)e_5-\sin(\theta)e_6)$, using $\frac{\partial s}{\partial r}(1,1)=4/3$. 

Combining these calculations with the identification in \autoref{e02.5} we conclude that $\{v_1,v_2,v_3\}$ provides a frame for $N_{Z/M}$. We claim this is consistent with the framing of $N_{Z/M}$ determined by the spin structure on $M$. The frame we computed came from a frame for $TM|_Z$, which itself was computed using the group structure on $\textup{Sp}(1)=S^3$. The spin structure on $M$ is inherited from the spin structure on $S^3$ and a Lie framing of $S^3$ (that is, one constructed using the group structure) extends the (unique) spin structure on $S^3$, viewed as a framing over the 2-skeleton \cite{KM99}. Thus, the framing for $N_{Z/M}$ determined by $\{v_1,v_2,v_3\}$ must agree with that induced by the spin structure on $M$. On the other hand, observing that $T_{p_\theta}H$ is spanned by $e_5,e_6,e_7,e_8$, we get a frame for $N_{Z/H}$ given by $w_1=e_7$, $w_2=e_8$, and $w_3=-\cos(\theta)e_5+\sin(\theta)e_6$. This frame, together with the frame $-v_0$ for $TZ$, extends over the obvious disc in $H$ bounding $Z$ and thus must agree with the framing on $N_{Z/H}$ induced by the spin structure on $H$. It now remains only to compare our frames using the isomorphism between $N_{Z/M}$ and $N_{Z/H}$ induced by the standard $G_2$-structure $\varphi_0$ on $\textup{Im}(\mathbb{O})$. We compute
\begin{equation}
    (\varphi_0||_Z)\;\lrcorner\; v_0 = \sin\theta (v^\ast_1\wedge w^\ast_2 - v^\ast_2\wedge w^\ast_1)-\cos\theta (v^\ast_1\wedge w^\ast_1+v^\ast_2\wedge w^\ast_2) -\frac{3}{4}v^\ast_3\wedge w^\ast_3
\end{equation}
and combining this with \autoref{l03.1} we see that the parity of $\mu$ is determined by whether or not the frames $\{v_1,v_2,v_3\}$ and $\{\sin\theta v_2+\cos\theta v_1, \cos\theta v_2-\sin\theta v_1, 3/4 v_3\}$ for $N_{Z/M}$ are homotopic. But these are not homotopic because the obvious map taking one frame to the other (after normalizing them to be orthonormal frames) represents the nontrivial element in $\pi_1(\textup{SO}(3))\cong\mathbb{Z}_2$. Therefore, $\mu=-1$.
\end{proof}


Although $M^+_{\epsilon,c}$ is not spin, we will see in the next section that there is a sort of relative sign that can be attached to the intersection circle of $M^+_{\epsilon, c}$ and $H_{s\epsilon}$, because of the existence of preferred bounding discs (see \autoref{f4.3}). In general, this will agree with the invariant $\mu$ when it is defined. We will regard $M^+_{\epsilon, c}$ and $H_{s\epsilon}$ as a sort of prototypical local model for intersection circles of sign $-1$. Note that, as already observed in \autoref{l3.8}, by deforming $H_{s\epsilon}$ from $s>s_0(c)$ to $s<s_0(c)$ we can remove the intersection circle with $M^+_{\epsilon,c}$ as pictured in \autoref{f6} below.

\begin{figure}[h]
\begin{center}
\begin{tikzpicture}[scale=1.25, decoration = {markings}]
\draw plot [smooth, tension=1.5] coordinates { (-1.5,2.5) (0,0) (1.5,2.5)};
    \node[label=left: ${\Sigma^+\subset M^+_{\epsilon,c}}$] at (-1.5, 2.75) {};
\draw plot [smooth, tension =1] coordinates { (-1.25,1.39) (0,1) (1.25,1.39)};
\draw[dashed] plot [smooth, tension=1] coordinates { (-1.25,1.39) (0,1.5) (1.25,1.39) };
\draw plot coordinates {(-3,1) (3,1)};
\draw plot coordinates {(-3,1) (-2.5,1.5)};
\draw[dashed] plot coordinates {(-2.5,1.5) (3.5,1.5)};
\draw plot coordinates {(3.5,1.5) (3,1)};
    \draw plot coordinates {(-3,0.5) (3,0.5)};
    \draw plot coordinates {(-3,0.5) (-2.5,1)};
    \draw[dashed] plot coordinates {(-2.5,1) (3.5,1)};
    \draw plot coordinates {(3.5,1) (3,0.5)};
\node[label=left: ${H_{(s_0(c)+1)\epsilon}}$] at (5.5,1.5) {};
\node[label=left: ${H_{(s_0(c)+0.5)\epsilon}}$] at (5.5,0.75) {};
\end{tikzpicture}   
\end{center}
\caption[Removing the intersection circle.]{Coassociative translations of $H_{s\epsilon}$ so that its intersection circle with $M^+_{\epsilon,c}$, contained in the fibre $\Sigma^+$, shrinks and disappears.}
\label{f6}
\end{figure}

The intersection circle of $M^-_{\epsilon, c}$ and $H_{s\epsilon}$ can also be deformed away by moving $s>0$ to $s\leq 0$, but this coassociative deformation of $H_{s\epsilon}$ does not remove the intersection circle ``along a disc.'' We regard this as a ``non-local'' separation of $M^-_{\epsilon, c}$ and $H_{s\epsilon}$, and by design \autoref{t1.1} excludes this case. We will discuss the motivation behind this further in \autoref{sl}. In the following section, we prove \autoref{t1.1}.

\subsection{Obstruction to separability}\label{s2.3}
Let $X_\pm\subset Y$ be transversely-intersecting coassociative submanifolds with compact intersection $Z$. Inspired by the Harvey-Lawson examples, we make the following definition. 

\begin{dfn}\label{d02.13}
We will say $\Gamma\in\pi_0(Z)$ is locally removable if there are smooth coassociative deformations $F_\pm : X_\pm\times [0,1]\rightarrow Y$, with respect to a smooth family of $G_2$-structures $\varphi_t$, and embeddings $\iota_\pm : D^2\rightarrow X_\pm$ such that
\begin{enumerate}
    \item $\partial D_\pm = F^{-1}_\pm(\Gamma)\subset X_\pm\times \{0\}$ where $D_\pm:=\iota_\pm(D^2)$.
    \item $F_\pm(X_\pm\times \{t\})$ intersect (transversely when $t\neq 1$) in $F_+(\iota_+(S_{1-t})\times \{t\})=F_-(\iota_-(S_{1-t})\times \{t\})$ for all $t\in [0,1]$, where $S_{1-t}=\{(1-t)e^{i\theta}\in D^2\; |\; \theta\in [0,2\pi)\}$.
\end{enumerate}
\end{dfn}

By Property 1 of \autoref{d02.13}, a removable component $\Gamma$ is necessarily contractible in both $X_+$ and $X_-$, and comes with preferred bounding discs $D_\pm$. These discs determine framings for $N_\pm|_\Gamma$, which can be compared using the isomorphism between $N_\pm$ determined by the $G_2$-structure.


\begin{lem}\label{p4.4.5}
If $\Gamma\in \pi_0(Z)$ is locally removable, then $D_+$ and $D_-$ must determine opposite framings for $N_+|_\Gamma$ and $N_-|_\Gamma$, respectively.
\end{lem}
\begin{proof}
The discs $D_+$ and $D_-$ determine a $2$-sphere in $X_+\#_\Gamma X_-$, as in the proof of \autoref{t4.11}, where the generalized connected sum is formed using the framings for $N_\pm|_\Gamma$ determined by $D_\pm$ and the isomorphism determined by the $G_2$-structure. By \autoref{p02.6}, the topology of $X_+\#_{\Gamma(t)}X_-$ is the same for all $t\in [0,1)$, thus since $\Gamma$ is removable it must be possible to blow down the $2$-sphere formed by $D_\pm$ (that is, the boundary of a neighborhood of this $2$-sphere in $X_+\#_\Gamma X_-$ must be homeomorphic to $S^3$). But this is only possible if the self-intersection of this $2$-sphere is $\pm 1$, which by the proof of \autoref{t4.11} forces the parity of $\Gamma$ to be $-1$.
\end{proof}

Note that so far we have not assumed $X_\pm$ are spin.

\begin{exm}\label{e02.16}
When $X_+=M^+_{\epsilon,c}$ and $X_-=H_{s\epsilon}$, for $s>s_0(c)$, we have discs $D_\pm$ along which the intersection circle vanishes for the family of coassociative deformations pictured in \autoref{f6}. Here, $D_+$ is swept out in the fibre $\Sigma^+$ of $X_+$, which does not move, and $D_-$ is the obvious disc in the $4$-plane $H_{s\epsilon}$ bounding the intersection circle. It follows from \autoref{p4.4.5} that the framings determined by these discs are opposite.
\end{exm}

The following proposition follows easily given \autoref{p4.4.5} and is the second statement in \autoref{t1.1}, the first statement having been proven already in \autoref{p02.6}.

\begin{prop}\label{p02.17}
Let $X_\pm$ be equipped with spin structures. If a component $\Gamma\in \pi_0(Z)$ is locally removable then
\begin{equation*}
    \mu_{X_+,X_-}(\Gamma)=-1.
\end{equation*}
\end{prop}
\begin{proof}
Given \autoref{p4.4.5}, it only remains to observe that the framings for $N_\pm|_\Gamma$ determined by any choice of bounding discs $D_\pm$ for $\Gamma$ must agree with the framings determined by the spin structures on $X_\pm$. The framing for $N_\pm|_\Gamma$ comes from the framing for $TX_\pm|_\Gamma$ determined by $D_\pm$, which we view as the data of a lift $\tilde{f}_\pm: \Gamma\rightarrow B\{1\}$ of the classifying map $f_\pm: \Gamma\rightarrow B\textup{SO}(4)$ for $TX_\pm|_\Gamma$. Likewise, the spin structure on $N_\pm|_\Gamma$ is induced by the one on $TX_\pm|_\Gamma$, which is given by a lift $f'_\pm: \Gamma\rightarrow B\textup{Spin}(4)$ of $f_\pm$. Since $H^1(S^1;\mathbb{Z}_2)=\mathbb{Z}_2$, there are in general two possible such lifts of the map $f_\pm$ up to (vertical) homotopy. The claim is that the lift of $f_\pm$ given by $Bi\circ \tilde{f}_\pm: \Gamma\rightarrow B\textup{Spin}(4)$, where $i:\{1\}\hookrightarrow \textup{Spin}(4)$ denotes the inclusion, is (vertically) homotopic to the lift $f'_\pm$. But this is necessarily the case because both extend to maps $D_\pm\rightarrow B\textup{Spin}(4)$.
\end{proof}

\subsection{Relation to near-symplectic geometry}\label{s02.5}
In this section, we relate our invariant $\mu$ to a construction from near-symplectic geometry in the special case when the coassociatives $X_\pm$ are graphical near their intersection $Z$. First, some generalities. If $X\subset (Y,\varphi)$ is a smooth coassociative, contracting normal vector fields with the $G_2$-structure $\varphi$ induces a bundle isomorphism
\begin{equation}\label{e2.4}
   N_{X/Y}\rightarrow \Lambda^+_X, \;\;\; V\mapsto (\varphi\;\lrcorner\; V)|_X
\end{equation}
where $\Lambda^+_X$ is the bundle of self-dual 2-forms on $X$ \cite{M98}. Thus, submanifolds sufficiently $C^1$-close to $X$ are identified with graphs of self-dual 2-forms. Suppose $X_\pm\subset Y$ are transversely-intersecting coassociatives with compact intersection $Z$ and that $X_+$ can be described near $Z$ as the graph of a self-dual $2$-form $\omega$ on $X_-$. Since $\omega$ vanishes on $Z$, we have an intrinsic derivative $\nabla\omega$, independent of the choice of connection $\nabla$. Contracting with a unit tangent field $u$ to $Z$ gives a section
\begin{equation}\label{e02.4}
    B_{X_+,X_-}:=\iota_u\nabla\omega\in \Gamma(Z,N^\ast_-\otimes N^\ast_-).
\end{equation}


\begin{lem}\label{p4.13}
$B_{X_+,X_-}$ is a nondegenerate, symmetric bilinear form satisfying
\begin{equation*}
    \tr B_{X_+,X_-}=\det B_{X_+,X_-}.
\end{equation*}
\end{lem}
\begin{proof}
Since $\omega$ vanishes transversely, $B_{X_+,X_-}$ is nondegenerate. Choose an orthogonal frame $e_i$ and coframe $\sigma^i$, for $0\leq i\leq 3$, centered at a point $p\in Z$, such that $u=e_0$. In this frame, let $\omega=T_{ij}(\sigma^i\wedge\sigma^j)$ and observe that for any $1\leq k\leq 3$ we have
\begin{equation*}
\nabla_{e_k}(T_{ij}(\sigma^i\otimes\sigma^j))=(D_{e_k}T_{ij})(\sigma^i\otimes\sigma^j)+T_{ij}(\Gamma^\ell_{kj}(\sigma^\ell\otimes\sigma^j)+ \Gamma^\ell_{ki}(\sigma^i\otimes\sigma^\ell)).
\end{equation*}
Since $\omega$ vanishes identically on $Z$, we know $T_{ij}(p)=0$ and thus $\nabla_{e_k}\omega = (D_{e_k}T_{ij})(\sigma^i\wedge\sigma^j)$. Contracting with $e_0$, we obtain $B_{X,X'}(e_k)=(D_{e_k}T_{0j})\sigma^j$ and in particular
\begin{equation}\label{e4.4}
    B_{X,X'}(e_k)e_m = D_{e_k}T_{0m}, \;\;\; B_{X,X'}(e_m)e_k=D_{e_m}T_{0k}.
\end{equation}
The trace of $B_{X_+,X_-}$ at the point $p$ is given by
\begin{equation}\label{e03.2}
    \tr(B_{X,X'})_p=(D_{e_k}T_{0j})g^{jk}=(D_{e_k}T_{0j})\delta^{jk}=D_{e_k}T_{0k}.
\end{equation}
Now, write $\omega=f_j\omega_j$ for $\omega_j$ the standard basis of self-dual $2$-forms written using the coframe $\sigma^i$. The condition that the graph of $\omega$ is coassociative is equivalent to
\begin{equation}\label{e4.5}
    df_1\wedge df_2\wedge df_3 =\sum df_j\wedge \omega_j,
\end{equation}
which is a system of four equations. Since $\omega$ vanishes identically on $Z$, we have $D_{e_0}f_j=0$ for all $1\leq j\leq 3$, and thus \autoref{e4.5} implies the three equations
\begin{equation*}
    D_{e_2}f_1=D_{e_1}f_2, \;\;\; D_{e_3}f_1=D_{e_1}f_3, \;\;\; D_{e_3}f_2=D_{e_2}f_3.
\end{equation*}
In light of \autoref{e4.4}, this says that $B_{X_+,X_-}$ is symmetric. The fourth equation comes from examining the coefficient on $df_1\wedge df_2\wedge df_3$ in \autoref{e4.5} and reads
\begin{equation*}
    \det (B_{X_+,X_-})_p = D_{e_1}f_1 + D_{e_2}f_2 + D_{e_3}f_3 =  \tr(B_{X_+,X_-})_p,
\end{equation*}
where the last equality uses \autoref{e03.2}.
\end{proof}

If $X_\pm$ are sufficiently $C^1$-close near $Z$, then $\det B_{X_+,X_-}$ is almost zero and so $B_{X_+,X_-}$ is almost traceless by \autoref{p4.13}. Thus, it has pointwise either one negative eigenvalue and two positive eigenvalues or vice versa. Reversing the orientation of $Z$ if necessary, we can assume $B_{X_+,X_-}$ has pointwise one negative and two positive eigenvalues. 

Denote by $L_{X_+,X_-}\subset N_-$ the  eigensubbundle corresponding to the negative eigenvalue. In this case, one can attach signs to the components of $Z$ via the assignment
\begin{equation}\label{e02.11}
    \Gamma\in \pi_0(Z)\mapsto w_1(L_{X_+,X_-}|_\Gamma)\in H^1(\Gamma;\mathbb{Z}_2).
\end{equation}

The map in \autoref{e02.11} imitates the construction of the so-called parity invariant in near-symplectic geometry, which we briefly digress to explain. Given a near-positive $2$-form $\omega$ on an oriented, smooth $4$-manifold $X$, one can find a conformal structure on $X$ for which $\Lambda^+_X$ has $\omega$ as a section \cite{ADK}. There is an associated bilinear form
\begin{equation*}
    B_\omega := \iota_u\nabla\omega\in N^\ast_{Z_\omega/X}\otimes N^\ast_{Z_\omega/X},
\end{equation*}
where $N_{Z_\omega/X}$ denotes the normal bundle to $Z_\omega=\omega^{-1}(0)\subset X$ and $u$ is a unit tangent vector field on $Z_\omega$. Because $\omega$ vanishes transversely, $B_\omega$ is nondegenerate. It is independent of the choice of conformal structure because $\Lambda^+_X|_Z=\textup{im}(\nabla\omega)$ for any such choice \cite{P07}. If $\omega$ is closed, and so near-symplectic, Taubes observed \cite{T} that $B_\omega$ is symmetric and $\tr B_\omega=0$. Orienting $Z_\omega$ so that $\det B_\omega<0$, one can attach signs to $Z_\omega$ using the eigensubbundle $L\subset N_{Z_\omega/X}$ corresponding to the negative eigenvalue \cite{P07}.

The fact that we can imitate this construction in our case reflects the fact that the self-dual $2$-form on $X_-$ whose graph is identified with $X_+$ is closed ``up to a small error'' when $X_\pm$ are sufficiently $C^1$-close near $Z$. A somewhat more surprising observation is that the signs from \autoref{e02.11} admit a different description coming from special Lagrangian geometry. Let $X_\pm$ be transverse coassociatives with compact intersection $Z$ and as before denote by $N_\pm$ the normal bundle to $Z$ in $X_\pm$ and by $N$ the normal bundle to $Z$ in $Y$.

\begin{prop}\label{p4.5}
$N_\pm$ form a transverse pair of special Lagrangian subbundles of $N$.
\end{prop}
\begin{proof}
We first show the $G_2$-structure $\varphi$ on $Y$ induces a reduction of the structure group of $N$ to $\textup{SU}(3)$. The bundle $\Lambda^3 T^\ast Y|_Z$ splits as $\Lambda^3 T^\ast Y|_Z=(T^\ast Z\otimes \Lambda^2 N^\ast)\oplus \Lambda^3 N^\ast$. Let $\Omega'$ denote the component of $\varphi||_Z$ in $\Lambda^3 N^\ast$, which is nowhere-vanishing because $\varphi$ is nondegenerate. Given a choice $u$ of unit tangent vector field to $Z$, we can repeat the construction in the proof of \autoref{l03.1} to define $\omega':=(\varphi||_Z)\;\lrcorner\;u\in \Gamma(Z,\Lambda^2 N^\ast)$, which we recall is nondegenerate because $\varphi$ is strongly nondegenerate. We can thus reduce the structure group of the frame bundle of $N$ to $\textup{SU}(3)\subset G_2$ by defining the principal $\textup{SU}(3)$ subbundle $\textup{Fr}_\varphi(N)\subset \textup{Fr}(N)$ with fibres
\begin{equation*}
    \textup{Fr}_\varphi(N)_y=\{u\in \textup{Fr}(N)_y\;\;|\;\; u^\ast\omega_0=\omega_y, \; u^\ast\textup{Im}(\Omega_0)=\Omega'_y\}
\end{equation*}
for all $y\in Y$ (recall that the stabilizer of a (nonzero) vector in $\mathbb{R}^7$ under the action of $G_2$ is isomorphic to $\textup{SU}(3)$ \cite{B05}). Note that the corresponding almost-complex structure on $N\rightarrow Z$ is given by taking the cross-product with $u$. Now, since $X_\pm$ intersect transversely, $N_\pm$ form a transverse pair of subbundles of $N$. Write $\varphi||_Z = u^\ast\wedge\omega'+\Omega'$. Considering the splitting $TX_{\pm}|_Z=TZ\oplus N_\pm$, one sees that $\omega'$ and $\Omega'$ must vanish on $N_\pm$ because $X_\pm$ are coassociative, meaning $\varphi|_{X_\pm}=0$. This completes the proof.
\end{proof}

We now show that the first Stiefel-Whitney class $w_1(L_{X_+,X_-})$ of the line bundle $L_{X_+,X_-}\rightarrow Z$ is an invariant of $N_\pm$ up to homotopy through transverse pairs of special Lagrangian subbundles that are graphical and sufficiently close.

\begin{rem}\label{r2.18}
If there were a nontrivial homotopy invariant for an arbitrary pair of transverse special Lagrangian subbundles $N_\pm\subset N$, then we could use it to define an invariant for $X_\pm$ without additional assumptions on their geometry. We will show (see \autoref{t3.11}) that there is no such homotopy invariant. This reflects the fact that $L_{X_+,X_-}$ is only well-defined when $X_\pm$ are sufficiently $C^1$-close near $Z$.
\end{rem}

Denote by $\textup{SLG}(n)$ the Grassmanian of Special Lagrangian planes in $\mathbb{C}^n$ for $n\geq 2$. The group $\textup{SU}(n)$ acts transitively and $\textup{SLG}(n)$ is diffeomorphic to $\textup{SU}(n)/\textup{SO}(n)$. There is an open submanifold
\begin{equation*}
    M_0(n) := \{(V,V')\;\; | \;\; V,V'\in \textup{SLG}(n), \; V\cap V'=\{0\}\},
\end{equation*}
of $(\textup{SLG}(n))^2$ consisting of the transverse planes. For an $n$-tuple $\theta=(\theta_1,\cdots,\theta_n)\in (0,\pi)^n$, denote by $V_\theta$ the corresponding $n$-plane $V_\theta = \{(e^{i\theta_1}x_1,\cdots, e^{i\theta_n}x_n)\;\; | \;\;x_j\in\mathbb{R}\}\subset \mathbb{C}^n$ and given a vector $\theta\in (0,\pi)^n$, denote by $\tr \theta=\theta_1+\cdots +\theta_n$ the sum of its entries. Each $V_\theta$ is Lagrangian and $V_\theta$ is special Lagrangian if and only if $\tr\theta=m\pi$ for $m\in \{1,\cdots, n-1\}$. In this case, $(V_0,V_\theta)\in M_0(n)$. Define the set
\begin{equation*}
    I_0(n) = \{\theta\in (0,\pi)^n \;\; | \;\; \theta_1\leq \cdots\leq\theta_n \;\;\textup{and}\;\;\tr\theta\in \pi\mathbb{Z}\}.
\end{equation*}

\begin{lem}[\cite{J03}, Prop. 9.1]\label{p4.3}
If $(V,V')\in M_0(n)$, there a unique $\theta\in I_0(n)$ for which there exists some $B\in \textup{SU}(n)$ satisfying $B\cdot (V,V')=(V_0,V_\theta)$.
\end{lem}

The natural map $I_0(n)\longrightarrow M_0(n)$ assigning $\theta$ to $(V_0,V_\theta)$ is continuous with respect to the subspace topologies on $I_0(n)$ and $M_0(n)$, and it follows from \autoref{p4.3} that it induces a homeomorphism
\begin{equation}\label{e03.6}
    I_0(n)\rightarrow M_0(n)/\textup{SU}(n).
\end{equation}

We now specialize to $n=3$. Denote by $K_4\cong \mathbb{Z}_2\times \mathbb{Z}_2$ the Klein 4-group and recall the set of diagonal matrices $D\subset \textup{SO}(3)$ is isomorphic to $K_4$. A straightforward calculation using the condition $\tr\theta\in\pi\mathbb{Z}$ gives the following result.

\begin{lem}\label{l4.5}
Let $(V_0,V_\theta)\in M_0(3)$ for $\theta\in I_0(3)$ and let $\textup{Stab}_{\textup{SU}(3)}(V_0,V_\theta)\subset \textup{SU}(3)$ be its stabilizer. We have the following possibilities
\begin{equation*}
   \textup{Stab}_{\textup{SU}(3)}(V_0,V_\theta)\cong \left\{\begin{matrix}
\textup{SO}(3) & \textup{if}\; \theta_1=\theta_2=\theta_3 \\
\textup{O}(2) & \textup{if}\; \theta_1=\theta_2<\theta_3\;\;\textup{or}\;\; \theta_1< \theta_2=\theta_3  \\
K_4 & \textup{if}\; \theta_1<\theta_2<\theta_3. \\
\end{matrix}\right.
\end{equation*}
\end{lem}
Using \autoref{p4.3}, we see \autoref{l4.5} completely characterizes the possible stabilizer subgroups of an arbitrary pair $(V,V')\in M_0(3)$. Now, for each $0\leq m\leq 3$, let $\Pi_m\subset I_0(3)$ denote the set of vectors $\theta\in I_0(3)$ for which exactly $m+1$ entries in the vector coincide. We have a decomposition of $I_0(3)$ into disjoint subspaces as the union of the $\Pi_m$, with $\Pi_0$ being the interior of $I_0(3)$, corresponding to the stabilizer subgroups of a pair $(V,V')\in M_0(3)$ as classified by \autoref{l4.5}. Define the orbit spaces $ \mathcal{O}_m := \textup{SU}(3)\cdot P(\Pi_m)\subset M_0(3)$, each of which is a finite union of submanifolds of $M_0(3)$.

\begin{prop}\label{p02.27}
$M_0(3)$ has two connected components and each is simply-connected.
\end{prop}
\begin{proof}
The homeomorphism in \autoref{e03.6} implies $M_0(3)/\textup{SU}(3)$ has two connected components and all its higher homotopy groups vanish. The generic region of $M_0(3)$ admits a fibration
\begin{equation*}
    \textup{SU}(3)/K_4\rightarrow (M_0(3)-\mathcal{O}_1-\mathcal{O}_2)\rightarrow (M_0(3)-\mathcal{O}_1-\mathcal{O}_2)/\textup{SU}(3).
\end{equation*}
Because each connected component of $\mathcal{O}_2$ is a smooth submanifold of $M_0(3)$ diffeomorphic to $\textup{SU}(3)/\textup{SO}(3)$, which has real dimension $5$, removing it does not affect the fundamental group of $M_0(3)$. Examining the long exact sequence in homotopy groups for the above fibration then gives that the fundamental group of either connected component of $M_0(3)-\mathcal{O}_1$ is isomorphic to $K_4$. Each of the two generators can be killed by explicit nullhomotopies connecting to points in $\mathcal{O}_1$ whose projections down to $I_0(3)$ lie, respectively, in one of the two boundary components (see \autoref{f2}). Thus, each component of $M_0(3)$ is simply-connected.
\end{proof}

\begin{thm}\label{t3.11}
Let $E\rightarrow Z$ be a vector bundle with structure group $\textup{SU}(3)$ over a compact 1-manifold $Z$. Any two pairs of transverse special Lagrangian subbundles of $E$ are homotopic through transverse pairs of special Lagrangian subbundles of $E$.
\end{thm}
\begin{proof}
 Let $\textup{F}(E)\rightarrow Z$ denote the bundle of oriented, orthonormal frames in $E$. We have the following associated fibre bundle
\begin{equation}\label{e02.14}
    M_0(E):=F(E)\times_{\textup{SU}(3)} M_0(3)
\end{equation}
and a pair $(V,V')$ of transverse special Lagrangian subbundles of $E\rightarrow Z$ determines sections $\eta_{(V,V')}$ and $\eta_{(V',V)}$ of $M_0(E)$. Pulling back the tautological bundles over $M_0(E)$ via these sections recovers the original pair of subbundles and in general every section $\eta$ of $M_0(E)$ is of the form $\eta=\eta_{(V,V')}$ for some $(V,V')$. Two transverse pairs $(V,V')$ and $(W,W')$ of special Lagrangian subbundles of $E$ are homotopic if and only if their corresponding sections $\eta_{(V,V')}$ and $\eta_{(W,W')}$ are homotopic. There is a bijection \cite{W78} between (vertical) homotopy classes of sections of $M_0(E)$ and elements of the cohomology group $H^1(Z;\pi_1(M_0(3)))$, which vanishes by \autoref{p02.27}. This completes the proof.
\end{proof}

We can restrict our attention to transverse graphical pairs of special Lagrangian subbundles, which correspond to the subspace of $(V,V')\in M_0(3)$ satisfying $V^\perp\cap V'=\{0\}$. This is identified via the homeomorphism in \autoref{e03.6} with the complement in $I_0(3)$ of the two line segments 
\begin{equation*}
 \ell_\pi=\{\theta\in I_0(3)\;\; | \;\; \theta_3=\pi/2\}\;\;\;\textup{and}\;\;\;  \ell_{2\pi}=\{\theta\in I_0(3)\;\; | \;\; \theta_1=\pi/2\},
\end{equation*}
which has four connected components. The same argument as in the proof of \autoref{p02.27} shows two are simply-connected, while the two that form neighborhoods of the diagonal $\Delta$ (the striped regions in \autoref{f2}) have fundamental group $\mathbb{Z}_2$.

\begin{figure}[h]
\begin{center}
\begin{tikzpicture}[decoration={
    markings,
    mark=at position 0.5 with {\arrow{>}}}
    ]
\node[label=left:${\theta_{1}=\theta_{2}}$] at (-7.5,4.6) {};
\node[label=left:${\theta_{2}=\theta_{3}}$] at (-5,3.2) {};
\node[label=left:${\theta_{1}=\theta_{2}}$] at (-0.5,2) {};
\node[label=left:${\theta_{2}=\theta_{3}}$] at (3,1) {};
\node[label=left:${\tr\theta=\pi}$] at (3.2,3.2) {};
\node[label=left:${\tr\theta=2\pi}$] at (-3.5,5) {};
\draw[thick] (-9,3) to (-4,4);
\draw[thick] (-9,3) to (-7,5);
\draw[dashed] (-7,5) to (-4,4);
\draw[thick] (-1,0) to (0,3.5);
\draw[thick] (-1,0) to (3,2);
\draw[dashed] (0,3.5) to (3,2);
\filldraw[fill=lightgray,opacity=0.5,draw=none] (-9,3) to (-4,4) to (-7,5);
\filldraw[fill=lightgray,opacity=0.5,draw=none] (-1,0) to (0,3.5) to (3,2);
\draw[thick, postaction={decorate}] (-7,5) to (-7.5,3.3) {}; 
\draw[thick, postaction={decorate}] (3,2) to (-0.65,1.2) {}; 
\node[label=left:${\ell_{2\pi}}$] at (-7.2,3.8) {};
\node[label=right:${\ell_\pi}$] at (0.2,1.2) {};
\node[dot,label=above right:$\Delta$] () at (-4,4) {};
\node[dot,label=above left:$\Delta$] () at (0,3.5) {};
\path[name path=A, draw, dotted] (-7,5) to (-7.5,3.3);
\path[name path=B, draw, dotted] (-4,4) to (-4,4);
\tikzfillbetween[of=A and B]{pattern=north east lines};
\path[name path=C, draw, dotted] (3,2) to (-0.65,1.2);
\path[name path=D, draw, dotted] (0,3.5) to (0,3.5);
\tikzfillbetween[of=C and D]{pattern=north east lines};
\end{tikzpicture}
\captionof{figure}{The two connected components of $I_0(3)$, interchanged by the $\mathbb{Z}_2$-action $(V,V')\mapsto (V',V)$. The vertex corresponds to stabilizer $\textup{SO}(3)$, the edges to $\textup{O}(2)$, and the open regions to $K_4$. The lines $\ell_\pi$ and $\ell_{2\pi}$ are labeled on $I_0(3)$ with arrows. The diagonal $\Delta\subset \textup{SLG}(3)^2$ is also marked.}
\label{f2}
\end{center}
\end{figure}

As such, there is an obstruction class $\mathfrak{ob}\in H^1(S^1;\mathbb{Z}_2)$, which vanishes for two given pairs $(V,V')$ and $(W,W')$ if and only if they are homotopic through close graphical pairs of transverse special Lagrangian subbundles of $E\rightarrow S^1$.

\begin{lem}\label{l02.23}
Given two close graphical pairs $(V,V')$ and $(W,W')$ of transverse special Lagrangian subbundles, there are line subbundles $L_{V,V'}\subset V$ and $L_{W,W'}\subset W$ such that
\begin{equation*}
    \mathfrak{ob}((V,V');(W,W')) = w_1(L_{V,V'})-w_1(L_{W,W'}).
\end{equation*}
\end{lem}
\begin{proof}
There is a unique linear map $V\rightarrow V^\perp$ whose graph is $V'$, and composing with the isomorphism $V^\perp\cong V^\ast$, using that $V$ is Lagrangian, gives a bilinear form $B_{V,V'}: V\rightarrow V^\ast$ satisfying $\tr B_{V,V'}=\det B_{V,V'}$, since $V'$ is special Lagrangian. The closeness assumption implies that we can associate, by the same construction as described following \autoref{p4.13}, an eigensubbundle $L_{V,V'}\rightarrow S^1$. Likewise for the pair $(W,W')$. A homotopy between $(V,V')$ and $(W,W')$ through close graphical pairs induces a homotopy between the eigensubbundles and so $w_1(L_{V,V'})=w_1(L_{W,W'})$. Conversely, suppose $w_1(L_{V,V'})=w_1(L_{W,W'})$. Choose a homotopy $V(t)$ between $V$ and $W$ that restricts to a homotopy $L(t)$ between $L_{V,V'}$ and $L_{W,W'}$. Having fixed the eigensubbundles $L(t)$, we only have to specify the bilinear forms on $L(t)^\perp\subset V(t)$ and since $w_1(L(t))=w_1(L^\perp(t))$ we can choose such a family continuously. The graphs of these give a homotopy $V'(t)$ between $V'$ and $W'$ and by construction $(V(t),V'(t))$ is a close graphical pair for each $t\in [0,1]$.
\end{proof}

We return now to our invariant $\mu$. Let $X_\pm$ be a pair of transverse coassociatives. Recall that to define the bilinear form in \autoref{e02.4} only required $\omega\in \Omega^2_+(X_-)$ to be defined in a neighborhood of $Z$. Suppose however that $\omega$ is globally-defined. For any choice of disc $D_\Gamma\subset X_-$ bounding $\Gamma\in \pi_0(Z)$, we can consider the corresponding disc $\omega(D_\Gamma)\subset X_+$. As in \autoref{s2.3}, we can compare the corresponding framings for $N_\pm|_\Gamma$ using the isomorphism between $N_\pm$ determined by the $G_2$-structure. This parity does not depend on the choice of disc $D_\Gamma$ and if $X_\pm$ are spin it agrees with the invariant $\mu(\Gamma)\in \{\pm 1\}$. On the other hand, we have the sign associated to $\Gamma$ via the first Stiefel-Whitney class of the line bundle $L_{X_+,X_-}\rightarrow Z$ as in \autoref{e02.11}. These two notions of sign agree.

\begin{prop}\label{p02.24}
If $X_\pm$ are sufficiently $C^1$-close, then for all $\Gamma\in \pi_0(Z)$ one has
\begin{equation*}
    \mu(\Gamma) = w_1(L_{X_+,X_-}|_\Gamma).
\end{equation*}
\end{prop}
\begin{proof}
Suppose without loss of generality that $Z$ is connected. In the statement of the proposition, we can either understand $\mu$ to be defined by spin structures on $X_\pm$ or in the relative sense, by comparing the framings for $N_\pm$ determined by bounding discs $D\subset X_-$ and $\omega(D)\subset X_+$. The two perspectives agree, as explained in the proof of \autoref{p02.17}. By \autoref{p4.5}, the normal bundles $N_\pm\subset N$ form a transverse pair of special Lagrangian subbundles and they are a close graphical pair because $X_\pm$ are assumed to be sufficiently $C^1$-close. The line bundle $L_{X_+,X_-}$ constructed in the paragraph following \autoref{p4.13} is evidently the same as the line bundle $L_{N_+,N_-}$ from \autoref{l02.23}. Let $Q\subset M_0(3)$ denote the space of close graphical pairs of transverse special Lagrangian $3$-planes in $\mathbb{C}^3$ and fix a basepoint in one of its components. Combining \autoref{p02.27} with the long exact sequence in homotopy for the pair $(M_0(3),Q)$ we find an isomorphism
\begin{equation}\label{e02.16}
    \pi_2(M_0(3),Q)\cong \pi_1(Q)
\end{equation}
given by the boundary map. Choose a trivialization of $N$ and fix a reference pair $(\underline{V}_0,\underline{V}_\theta)$ of trivial bundles. Using the trivialization, the section $\eta_{(N_+,N_-)}$ of $M_0(N)$ induces a map $S^1\rightarrow Q$ whose class in $\pi_1(Q)\cong\mathbb{Z}_2$ has the same parity as the obstruction class $\mathfrak{ob}$, which is equal to $w_1(L_{N_+,N_-})$ by \autoref{l02.23}. On the other hand, the corresponding element in $\pi_2(M_0(3),Q)$ under the isomorphism of \autoref{e02.16} is given precisely by extending $N_-$ and $N_+$ over bounding discs $D$ and $\omega(D)$, respectively, and its parity is $\mu(Z)$.
\end{proof}

\section{Implications}\label{sl}
Our discussion in this concluding section is somewhat speculative and partially contingent on the main theorem of the forthcoming paper \cite{G25}. Nevertheless, we hope presenting this discussion now will help to clarify the relation of the invariant $\mu$, and the concept of local removability, to non-compactness phenomena for coassociatives. In \cite{G25}, more concrete corollaries of \autoref{t1.1} and \autoref{t1.2} will be established.

\subsection{Non-compactness phenomena}
Let $X\subset (Y,\varphi)$ be a smooth coassociative submanifold of a $G_2$-manifold. McLean \cite{M98} showed that if $d\varphi=0$, then a normal vector field $V$ to $X$ is the deformation vector field to a family of coassociative submanifolds of $Y$ if and only if $d\omega_V=0$, where $\omega_V:=(\varphi\;\lrcorner\; V)|_X$ is the self-dual $2$-form on $X$ corresponding to $V$ under the isomorphism in \autoref{e2.4}. Using this, McLean further showed the following result.

\begin{thm}[\cite{M98}, Theorem 4.5]\label{t04.1}
Let $(Y,\varphi)$ be a $G_2$-manifold with $d\varphi=0$. Let $X\subset Y$ be a smooth, compact coassociative submanifold. The moduli space $\mathcal{M}_{\textup{co-def}}(X)$ of coassociative deformations of $X$ is a smooth manifold of dimension $b^2_+(X)$.
\end{thm}

Here, $b^2_+(X)$ is the dimension of the space of closed self-dual $2$-forms on $X$ and $\mathcal{M}_{\textup{co-def}}(X)$ consists of equivalence classes of coassociative embeddings $X\rightarrow Y$ isotopic through coassociative embeddings to $X\subset Y$, where two such are equivalent if they differ by a diffeomorphism of $X$. In general, $X$ might degenerate in a family to develop singularities and so $\mathcal{M}_{\textup{co-def}}(X)$ will be non-compact. If $\varphi$ is both closed and co-closed, then using the notion of $\varphi$-positive currents \cite{HL82} one can take the closure 
\begin{equation}\label{e04.1}
    \overline{\mathcal{M}}_{\textup{co-def}}(X)
\end{equation}
of $\mathcal{M}_{\textup{co-def}}(X)$ inside the space of integral currents in $Y$. This is compact because the coassociatives in $Y$ isotopic to $X$ have uniformly bounded volume, being calibrated by $\ast_\varphi\varphi$ (c.f. \cite{J03}). A central question in coassociative geometry is to characterize the ``boundary''
\begin{equation}\label{e04.2}
    \partial \overline{\mathcal{M}}_{\textup{co-def}}(X):=  \overline{\mathcal{M}}_{\textup{co-def}}(X)\setminus \mathcal{M}_{\textup{co-def}}(X),
\end{equation}
which roughly speaking means to understand all possible singular degenerations of $X$. By work of Lotay \cite{L09}\cite{L14}, it is known for instance that coassociatives can degenerate in one-parameter families to develop isolated conical singularities. We wish to consider the case when $X$ degenerates into a transverse intersection. 

Let us return for the moment to the Harvey-Lawson coassociatives discussed in \autoref{s2.4}. Although these are non-compact, the deformation theory of asymptotically conical (and conical singular) coassociatives is well-understood due to work of Lotay \cite{L07}\cite{L090}. Let $X_+=M^+_{\epsilon,c}$ for $(\epsilon,c)\in\mathcal{C}^+$ and $X_-=H_{s\epsilon}$. Fixing $\epsilon$, $X_+$ has an obvious family of coassociative deformations given by varying the parameter $c>0$. Similarly, $X_-$ has an obvious family of coassociative deformations given by varying $s$.

By work of Lotay and Kapouleas \cite{LK17}, one can desingularize the singular coassociative $M^+_{\epsilon,c}\cup H_{s\epsilon}\subset \mathbb{R}^7$ for $s>s_0(c)$ by gluing in a family of Lawlor necks parametrized by the intersection circle. Denote by $X_\tau(c,s)\subset \mathbb{R}^7$ the corresponding smooth coassociative, where the gluing parameter $0<\tau\ll 1$ corresponds to the scale of the Lawlor necks (the Lawlor neck in $\mathbb{C}^3$ has a one-dimensional space of deformations as an asymptotically-conical special Lagrangian corresponding to rescaling \cite{M02}). There is an evident three-parameter family of coassociative deformations of $X_\tau(c,s)$ with parameter space
\begin{equation*}
    \mathcal{P}=\{(\tau,c,s)\; | \; \tau\in (0,\tau_0), c\in (0,\infty), s\in (s_0(c),\infty)\}.
\end{equation*}

Singular degenerations of $X_\tau(c,s)$ correspond to boundary points of $\mathcal{P}$. Define
\begin{equation*}
    \bar{\mathcal{P}}=\{(\tau,c,s)\; | \; \tau\in [0,\tau_0), c\in [0,\infty), s\in [s_0(c),\infty)\}
\end{equation*}
noting that $s_0(0)=0$ and let
\begin{equation*}
    \partial_{\{\tau=0\}}\mathcal{P}=\{(0,c,s)\in \bar{\mathcal{P}}\}.
\end{equation*}

\begin{figure}[h]
\begin{center}
\begin{tikzpicture}[decoration={
    markings,
    mark=at position 0.5 with {\arrow{>}}}
    ]
\node[dot,label=below left:${(0,0,0)}$] at (-7.5,-5) {};
\draw[thick] (-7.5,-5) to (2,-5);
\draw[dashed] (-7.5,-5) to (-6,-3);
\draw[dotted] (2,-5) to (3.5,-3);
\draw[dotted] (-6,-3) to (3.5,-3);
\draw[thick] (-7.5,-5) to (-7.5,-1);
\draw[dotted] (2,-5) to (2,-1);
\draw[dotted] (-7.5,-1) to (2,-1);
\filldraw[fill=lightgray,opacity=0.5,draw=none] (-7.5,-5) to (-7.5,-1) to (-6,-3);
\node[dot,label=below right:${(\tau,c,s)}$] at (-4,-3.6) {};
\draw[dotted, postaction={decorate}] (-4,-3.6) to (-7.1,-3.6) {}; 
\node[dot,label={[xshift=0.25cm, yshift=0.1cm]:${(0,c,s)}$}] at (-7.1,-3.6) {};
\draw[dotted,postaction={decorate}] (-7.1,-3.6) to (-7.5,-4.1);
\node[dot,label=left:${(0,0,s)}$] at (-7.5,-4.1) {};
\draw[dotted,postaction={decorate}] (-7.1,-3.6) to (-7.1,-4.5);
\node[dot,label=right:${(0,c,s_0(c))}$] at (-7.1,-4.5) {};
\end{tikzpicture}
\captionof{figure}{The shaded triangle corresponds to $\partial_{\{\tau=0\}}\mathcal{P}\subset\bar{\mathcal{P}}$. Degenerations in one-parameter families are suggested by dotted arrows.}
\label{f07}
\end{center}
\end{figure}

One can think of $\partial_{\{\tau=0\}}\mathcal{P}\subset\bar{\mathcal{P}}$ as corresponding to a region in $\overline{\mathcal{M}}_{\textup{co-def}}(X_\tau(c,s))$ near $X_0(c,s)\in \partial \overline{\mathcal{M}}_{\textup{co-def}}(X_\tau(c,s))$, and we can understand the stratification of $\partial_{\{\tau=0\}}\mathcal{P}$, depicted in \autoref{f07}, in terms of singular degenerations of $X_\tau(c,s)$, $M^+_{\epsilon,c}$, and $H_{s\epsilon}$. The point $(0,c,s)$ corresponds to the singular coassociative $M^+_{\epsilon,c}\cup H_{s\epsilon}$ obtained by letting $\tau\rightarrow 0^+$ and thus shrinking down to zero the Lawlor necks used to resolve $M^+_{\epsilon,c}\cup H_{s\epsilon}$. The point $(0,c,s_0(c))$ corresponds to the union of the two smooth coassociative submanifolds $M^+_{\epsilon,c}$ and $H_{s_0(c)\epsilon}$ intersecting in a point, which is obtained by keeping $M^+_{\epsilon,c}$ fixed while deforming $H_{s\epsilon}$ so that its circle intersection with $M^+_{\epsilon,c}$ shrinks to a point, as described in \autoref{s2.4}. The point $(0,0,s)$ corresponds to the union of $H_{s\epsilon}$ with $M^+_{\epsilon,0}$, intersecting in a circle, obtained by keeping $H_{s\epsilon}$ fixed while deforming $M^+_{\epsilon,c}$ so that it converges to the Lawson-Osserman cone $M^+_{\epsilon,0}$. Finally, the point $(0,0,0)$ in the bottom left corner of \autoref{f07} corresponds to $M^+_{\epsilon,0}\cup H_0$, for which a one-parameter family of smoothings is provided by the other Harvey-Lawson coassociative $M^-_{\epsilon,\tau}$ for $\tau>0$. Thus, the horizontal solid line in \autoref{f07} corresponds to the one-parameter family of coassociative deformations of $M^-_{\epsilon,\tau}$ given by varying $\tau$.

\begin{rem}
What we wish to emphasize with this example is that one consequence of the removability of the circle intersection of $M^+_{\epsilon,c}$ and $H_{s\epsilon}$ is the existence of an additional piece of $\partial \mathcal{M}_{\textup{co-def}}(X_\tau(c,s))$ corresponding to a degeneration of $X_\tau(c,s)$ to a conically-singular coassociative (whose singularity in this case is modeled on the cone $M^+_{\epsilon,0}\cup H_0$).     
\end{rem}

Suppose now that $X_\pm$ are compact coassociatives in a torsion-free $G_2$-manifold $(Y,\varphi)$. Assume their transverse intersection $Z$ is connected and nullhomologous in $X_\pm$, and let $X=X_\tau\subset Y$ for $0<\tau\ll 1$ be the smooth coassociative submanifold of $Y$ constructed in \cite{G25} be resolving $X_0:=X_+\cup X_-$ with a family of Lawlor necks at scale $\tau$ parameterized by $Z$. Note that $\mathcal{M}_{\textup{co-def}}(X_0)$ is not a priori well-defined because $X_0$ is singular. However, we can imitate Joyce's construction \cite{J04I} of the moduli space of deformations of a conically-singular special Lagrangian and make the following tentative definition.

\begin{dfn}\label{d04.2}
Define $\mathcal{M}_{\textup{co-def}}(X_0)$ as a set to be all pairs $(\hat{X}_+,\hat{X}_-)$ such that 
\begin{enumerate}
    \item $\hat{X}_\pm\subset Y$ are coassociative submanifolds intersecting transversely.
    \item There is a homeomorphism $\hat{\iota}:X_0\rightarrow \hat{X}_0$ such that $\hat{\iota}(Z)=\hat{Z}$ and $\hat{\iota}$ restricts to a diffeomorphism on the complement of $Z$, where $\hat{X}_0=\hat{X}_+\cup \hat{X}_-$ and $\hat{Z}=\hat{X}_+\cap \hat{X}_-$. 
    \item $\hat{\iota}|_{X_\pm}$ is isotopic through coassociative embeddings to the inclusion $\iota_\pm:X_\pm\hookrightarrow Y$. 
\end{enumerate}
\end{dfn}

Since transversality is an open condition, one could proceed as in \cite{J04I} to topologize $\mathcal{M}_{\textup{co-def}}(X_0)$ by defining basic open neighborhoods consisting of pairs of all sufficiently nearby coassociative submanifolds expressed as graphs of small self-dual $2$-forms. We will not pursue this further here but, again drawing on work of Joyce \cite{J03}, it seems reasonable that with respect to a suitable topology on $\mathcal{M}_{\textup{co-def}}(X_0)$ the compactified moduli space $\overline{\mathcal{M}}_{\textup{co-def}}(X)$ should be locally diffeomorphic to
\begin{equation}\label{e04.3}
    \mathcal{M}_{\textup{co-def}}(X_0)\times [0,\infty)
\end{equation}
near $\mathcal{M}_{\textup{co-def}}(X_0)\subset \partial \overline{\mathcal{M}}_{\textup{co-def}}(X)$, where the local diffeomorphism
\begin{equation*}
     \mathcal{M}_{\textup{co-def}}(X_0)\times [0,\infty)_\tau\rightarrow \overline{\mathcal{M}}_{\textup{co-def}}(X)
\end{equation*}
is given by the gluing construction in \cite{G25}. Denote by $\partial_0 \overline{\mathcal{M}}_{\textup{co-def}}(X)$ the connected component of $\partial \overline{\mathcal{M}}_{\textup{co-def}}(X)$ containing $X_0$. In analogy with the Harvey-Lawson coassociatives, we expect the geometry of $\partial_0 \overline{\mathcal{M}}_{\textup{co-def}}(X)$ depends on whether $Z$ is locally removable and therefore by \autoref{t1.1} on the parity of the invariant $\mu$.

\begin{conj}\label{c4.3}
Let $X_\pm$ be smooth, compact coassociative submanifolds of a torsion-free $G_2$-manifold $Y$. Suppose that $X_\pm$ are simply-connected and that their intersection $Z$ is connected. Let $X\subset Y$ denote the smooth coassociative obtained by resolving $X_+\cup X_-$ as in \cite{G25}. If $Z$ is locally removable, then $X$ admits a one-parameter family of coassociative deformations to a conically-singular coassociative with isolated singularity modeled on the Lawson-Osserman cone.
\end{conj}

We conclude with some evidence for \autoref{c4.3} in the case when $X_\pm$ are rigid, and thus $Z$ is assumed to be removable along a family of deformations $X_\pm(t)$ that are coassociative with respect to a family of $G_2$ structures $\varphi_t$. Let $X_\pm=S^4$. It is shown in \cite{G25} that $X$ is diffeomorphic to the generalized connected sum from \autoref{l03.1}. Since $Z$ is removable, \autoref{t1.1} implies $\mu(Z)=-1$. Thus, $X$ is diffeomorphic to $\mathbb{CP}^2\#\overline{\mathbb{CP}}^2$ by \autoref{t1.2}, so $\mathcal{M}_{\textup{co-def}}(X)$ is one-dimensional by \autoref{t04.1}. In light of \autoref{e04.3}, we expect $\overline{\mathcal{M}}_{\textup{co-def}}(X)$ should be diffeomorphic to a closed interval, with one endpoint corresponding to $X_0=X_+\cup X_-$. 

What about the other endpoint? If, consistent with \autoref{c4.3}, $X$ did admit a degeneration to a coassociative $X'$ with isolated singularity modeled on the Lawson-Osserman cone, then one could use the gluing construction of Lotay \cite{L14} to desingularize $X'$ by gluing in the Harvey-Lawson coassociative $M^+$. Gluing in $M^+$ has the effect topologically of taking the connected sum of $X'$ with $\overline{\mathbb{CP}}^2$. On the other hand, any coassociative $X'=\mathbb{CP}^2$ with an isolated singularity modeled on the Lawson-Osserman cone is rigid as a conically-singular coasociative \cite{L07}\cite{L12} and has a one-dimensional moduli space of coassociative smoothings \cite{L14}. Thus, everything is consistent with our dimension count and our calculation of the diffeomorphism type of $X$.

If \autoref{c4.3} were correct in this rigid case, then  $\mathcal{M}_{\textup{co-def}}(X)$ would be diffeomorphic to $(0,\infty)$ and we could view the scale $\tau$ as giving a global coordinate on the moduli space in the following way. We can think of $\tau$ as the size of the central $2$-sphere of the Lawlor neck, which shrinks to zero at $\tau\rightarrow 0^+$. Inspecting the pre-gluing construction in \cite{G25}, one sees that this $2$-sphere represents a nonzero homology class $\alpha\in H_2(X_\tau)$. On the other hand, the exceptional $2$-sphere in $\mathbb{CP}^2\#\overline{\mathbb{CP}}^2$, which shrinks to zero as one approaches the right endpoint of the interval in \autoref{f08}, represents another nonzero homology class $\beta\in H_2(X_\tau)$. Together, $\alpha$ and $\beta$ span $H_2(X_\tau)$ and satisfy $\alpha\cdot\beta=\pm 1$. Thus, approaching the right endpoint in \autoref{f08} indeed corresponds to the limit $\tau\rightarrow +\infty$.

\begin{figure}[h]
\begin{center}
\begin{tikzpicture}[decoration={
    markings,
    mark=at position 0.5 with {\arrow{>}}}
    ]
\node[dot,label=below right:$\mathbb{CP}^2$] at (2,-5) {};
\draw[thick] (-7.5,-5) to (2,-5);
\node[dot,label=below left:$S^4\cup_{S^1} S^4$] at (-7.5,-5) {};
\node[dot,label=above:{$X=X_\tau\cong S^4\#_{S^1} S^4$}] at (-4,-5) {};
\end{tikzpicture}
\captionof{figure}{The compactification of the moduli space $\mathcal{M}_{\textup{co-def}}(S^4\#_{S^1}S^4)$ when $S^1$ is locally removable predicted by \autoref{c4.3}.}
\label{f08}
\end{center}
\end{figure}

\begin{rem}
One can regard the phenomenon predicted by \autoref{c4.3} as a sort of ``coassociative blow-down'' as \autoref{c4.3} says essentially that the $2$-sphere with self-intersection $-1$ in $X$ formed by the discs along which the circle $Z$ is removed can be blown-down through a family of coassociative deformations of $X$. Note that if $Z$ is locally removable in a family $X_\pm(t)$, then we can perform the generalized connected sum construction of \cite{G25} along each of the intersection circles $Z(t)$, recalling \autoref{d02.13}, and this gives a one-parameter family $X(t)$ of coassociative deformations of the smooth coassociative $X$ resolving $X_+\cup X_-$. Thus, the point of \autoref{c4.3} is to show that in the limit this family $X(t)$ of smooth coassociatives degenerates to a singular coassociative with precisely a Lawson-Osserman cone singularity.
\end{rem}

Note \autoref{t1.1} implies the intersection circle in \autoref{c4.3} has sign $\mu(Z)=-1$. It is plausible the invariant $\mu$ could be used to rule out certain possible non-compactness phenomena for coassociatives. We hope to investigate this further in a future work.

\end{document}